\documentclass[12pt]{amsart}
\usepackage{tabu}
\usepackage{microtype, mathrsfs, xfrac, xspace}
\usepackage{amsmath, amssymb, enumerate, amsthm, url}
\usepackage{enumitem}
\usepackage{tikz}
\usetikzlibrary{arrows.meta,positioning,matrix}
\usepackage[all]{xy}

\usepackage[utf8]{inputenc}
\usepackage[english]{babel}

\usepackage{graphicx}
\usepackage[alphabetic]{amsrefs}
\usepackage{amsfonts}
\newtheorem{theorem}{Theorem}[section]
\newtheorem{lemma}[theorem]{Lemma}

\theoremstyle{definition}

\newtheorem{example}[theorem]{Example}
\newtheorem{proposition}[theorem]{Proposition}

\newtheorem{corollary}[theorem]{Corollary}
\newtheorem{question}[theorem]{Question}

\theoremstyle{remark}
\newtheorem{remark}[theorem]{Remark}

\newcommand{\ind}{\operatorname{ind}}

\newcommand{\bbC}{{\mathbb C}}
\newcommand{\bbP}{{\mathbb P}}
\newcommand{\bbG}{{\mathbb G}}
\newcommand{\bbF}{{\mathbb F}}
\newcommand{\bbZ}{{\mathbb Z}}
\newcommand{\bbQ}{{\mathbb Q}}
\DeclareMathOperator{\ed}{ed}
\DeclareMathOperator{\Br}{Br}
\DeclareMathOperator{\GL}{GL}

\DeclareMathOperator{\PGL}{PGL}

\DeclareMathOperator{\Alt}{A}
\DeclareMathOperator{\SO}{SO}
\DeclareMathOperator{\Tr}{Tr}
\DeclareMathOperator{\Spin}{Spin}
\DeclareMathOperator{\Orth}{O}
\DeclareMathOperator{\Sym}{S}
\DeclareMathOperator{\Gal}{Gal}

\DeclareMathOperator{\Char}{char}

\DeclareMathOperator{\Ker}{Ker}
\DeclareMathOperator{\Pin}{Pin}
\DeclareMathOperator{\Mat}{Mat}

\begin{document}

\title[]{Essential dimension of double covers of symmetric and alternating groups}

\author{Zinovy Reichstein}
\address{Department of Mathematics\\
 University of British Columbia\\
 Vancouver, BC V6T 1Z2\\Canada}
 \email{reichst@math.ubc.ca}
\thanks{Partially supported by
 National Sciences and Engineering Research Council of
 Canada Discovery grant 253424-2017.}

\author{Abhishek Kumar Shukla}
\email{abhisheks@math.ubc.ca}
\thanks{Partially supported by a graduate fellowship from the Science and Engineering Research Board, India.}

	
%

\begin{abstract} I. Schur studied double covers $\widetilde{\Sym}^{\pm}_n$ and $\widetilde{\Alt}_n$ of symmetric groups $\Sym_n$ and alternating groups $\Alt_n$, respectively.
Representations of these groups are closely related to projective representations of $\Sym_n$ and $\Alt_n$; there is also 
a close relationship between these groups 
and spinor groups. We study the essential dimension $\ed(\widetilde{\Sym}^{\pm}_n)$ and $\ed(\widetilde{\Alt}_n)$. We show that over a base field of characteristic $\neq 2$, $\ed(\widetilde{\Sym}^{\pm}_n)$ and $\ed(\widetilde{\Alt}_n)$ grow exponentially with $n$, similar to $\ed(\Spin_n)$. On the other case, in characteristic $2$, they grow sublinearly,
similar to $\ed(\Sym_n)$ and $\ed(\Alt_n)$. We give an application of our result in good characteristic to the theory of trace forms. 
\end{abstract}

\maketitle


\section{Introduction} 

I. Schur~\cite{schur} studied central extensions 
\begin{equation} \label{e.extension}
\xymatrix{ 1 \ar@{->}[r] &  \bbZ/ 2 \bbZ \ar@{->}[r] & \widetilde{S}_n^{\pm} \ar@{->}[r]^{\phi^{\pm}} &
\Sym_n \ar@{->}[r] &  1} 
 \end{equation} 
of the symmetric group $\Sym_n$. Representations of these groups are closely related to projective representations of $\Sym_n$: over an algebraically closed field of characteristic zero,
every projective representation
$\rho \colon \Sym_n \to \PGL(V)$ lifts to linear representations $\rho^{+} \colon \widetilde{\Sym}_n^+ \to \GL(V)$ 
and $\rho^{-} \colon \widetilde{\Sym}_n^- \to \GL(V)$; see  \cite{hoff}*{Theorem 1.3}. That is, the diagram 
\[ \xymatrix{\widetilde{\Sym}_n^{\pm} \ar@{->}[d]_{\phi^{\pm}} \ar@{->}[r]^{\rho^{\pm}} & \GL(V) \ar@{-}[d]  \\
\Sym_n \ar@{->}[r]^{\rho} & \PGL(V) } \]
commutes. Moreover, the groups $\widetilde{\Sym}_n^{\pm}$ are minimal
central extensions of $\Sym_n$ with this property. They are called representation groups of $\Sym_n$. 
In terms of generators and relations, 
\[ \widetilde{\Sym}_n^+=\left< z, s_1,s_2,\ldots, s_{n-1} \mid z^2 = s_i^2=1, [z,s_i]=1,(s_is_j)^2= z \text{ if } |i-j|>1, (s_i s_{i+1})^3=1 \right>   \]
and 
\[ \widetilde{\Sym}_n^-=\left< z,t_1,t_2,\ldots, t_{n-1} \mid z^2=1,t_i^2=z, (t_it_j)^2=z \text{ if } |i-j|>1, (t_i t_{i+1})^3=z \right> . \]
Here $z$ is a central element of order $2$ in $\widetilde{\Sym}_n^+$ (respectively, $\widetilde{\Sym}_n^-$) 
generating $\Ker(\phi^+)$
(respectively, $\Ker(\phi^{-})$), and $\phi^+(s_i) = \phi^{-}(t_i)$ is the transposition $(i, i + 1)$ in
$\Sym_n$. The preimage of $\Alt_n$ under $\phi^+$ in $\widetilde{\Sym}_n^+$ is isomorphic to the preimage of $\Alt_n$ under
$\phi^{-}$ in $\widetilde{\Sym}_n^-$; see~\cite{serre-topics}*{Section 9.1.3}. We will denote this group by $\widetilde{\Alt}_n$;
it is a representation group of $\Alt_n$.  For modern expositions of Schur's theory,
see~\cite{hoff} or~\cite{stembridge}.

The purpose of this paper is to study the essential dimension of the covering groups $\widetilde{\Sym}_n^{\pm}$ and 
$\widetilde{\Alt}_n$. We will assume that $n \geqslant 4$ throughout.
As usual, we will denote the essential dimension of a linear algebraic group $G$ by $\ed(G)$
and the essential dimension of $G$ at a prime $p$ by $\ed(G; p)$. These numbers depend on the base field $k$;
we will sometimes write $\ed_k(G)$ and $\ed_k(G; p)$ in place of $\ed(G)$ and $\ed(G; p)$
to emphasize this dependence. We refer the reader to Section~\ref{sect.prelim} for the definition 
of essential dimension, some of its properties and further references. 

Our interest in the covering groups $\widetilde{\Sym}_n^{\pm}$, $\widetilde{\Alt}_n$ was motivated by 
their close connection to two families of groups whose essential dimension was 
previously found to behave in interesting ways, namely permutation groups and spinor groups. 
The connection with spinor groups is
summarized in the following diagram. Here the base field $k$ is assumed to be of characteristic $\neq 2$ and to contain a primitive $8$th root of unity, $\Orth_n$ is the orthogonal group associated to the 
quadratic form $x_1^2 + \dots + x_n^2$, $\Sym_n \rightarrow \Orth_n$
is the natural $n$-dimensional representation, $\Gamma_n$ is the Clifford group, and $\widetilde{\Sym}_n^{\pm}$, 
$\widetilde{\Alt}_n$ are the preimages of $\Sym_n$ and $\Alt_n$ under the double covers
$\Pin_n^{\pm}\to \Orth_n$.
The groups $\Pin_n^{\pm}$ are defined as the kernels of the homomorphisms $N^{\pm}:\Gamma_n \to \mathbb{G}_m$ given by 
$N^+(x)=x.x^T$ and  $N^-(x)=x.\gamma(x^T)$, where 
 $(x_1\otimes x_2\otimes \ldots \otimes x_n)^T=x_n\otimes \ldots \otimes x_2 \otimes x_1 $ and $\gamma$ is the automorphism of the Clifford algebra which acts on degree $1$ component by $-1$.
 Both appear in literature as $\Pin$ groups ($\Pin^+$ in \cite{serre-trace} and $\Pin^-$ in \cite{abs}, see also \cite{gagola}). 

{\footnotesize
\begin{center}
\begin{tikzpicture}
\matrix (m) [matrix of math nodes,
row sep=1em, column sep=1em,
text height=1.5ex,
text depth=0.25ex]{
    & & & & \Orth_n &  \\
    & & \Gamma_n &  & &  \\
\mathbb{G}_m & & & & &  \\
    & \Pin_n^+ &  & & \Orth_n &  \\
    &  &  & &  &  \\
\bbZ / 2 \bbZ & & & \Pin_n^- & &  \\
    & \widetilde{\Sym}_n^+ & & &  \Sym_n& \\
    &  &  & &  &  \\
\bbZ/ 2 \bbZ & & & \widetilde{\Sym}_n^- & &  \\
    & & & & \Alt_n & \\
    & & \widetilde{\Alt}_n& & &  \\
\bbZ / 2\bbZ  & & & & &  \\
};
\path[-{Latex[length=2.5mm, width=1.5mm]}]
(m-2-3) edge (m-1-5)
(m-3-1) edge (m-2-3)
(m-6-1) edge (m-4-2)
        edge (m-6-4)
        edge (m-3-1)
(m-4-5) edge (m-1-5)
(m-4-2) edge (m-4-5)
(m-6-4) edge (m-4-5)
(m-4-2) edge (m-2-3)
(m-6-4) edge (m-2-3)
(m-9-1) edge (m-7-2)
        edge (m-9-4)
        edge (m-6-1)
(m-7-2) edge (m-7-5)
(m-9-4) edge (m-7-5)
(m-7-5) edge (m-4-5)
(m-7-2) edge (m-4-2)
(m-9-4) edge (m-6-4)
(m-12-1) edge (m-11-3)
        edge (m-9-1)
(m-11-3) edge (m-10-5)
        edge (m-9-4)
        edge (m-7-2)
(m-10-5) edge (m-7-5)
;
\end{tikzpicture}
\end{center} }

The essential dimension of $\Sym_n$ and $\Alt_n$ is known to be sublinear in $n$:
in particular, $n \geqslant 5$, we have 
\[ \ed(\Alt_n) \leqslant \ed(\Sym_n) \leqslant n - 3 ; \]
see~\cite{buhler-reichstein}*{Theorem 6.5(c)}.
On the other hand, the essential dimension
of $\Spin_n$ increases exponentially with $n$. If we write $n = 2^a m$, where $m$ is odd, then
\[ \ed(\Spin_n) = \ed(\Spin_n; 2) = \begin{cases} 
\text{$2^{ (n - 1)/2} - \dfrac{n(n - 1)}{2}$, if $a = 0$,} \\
\text{$2^{(n - 2)/2} - \dfrac{n(n - 1)}{2}$, if $a = 1$,} \\
\text{$2^{(n - 2)/2} + 2^a - \dfrac{n(n - 1)}{2}$, if $a \geqslant 2$;}
\end{cases} \]
see~\cite{brv-annals}, \cite{cher-mer} and~\cite{totaro-spin}.

\begin{question} \label{q1}
What is the asymptotic behavior of 
$\ed(\widetilde{\Sym}_n^{\pm})$, $\ed(\widetilde{\Alt}_n)$,
$\ed(\widetilde{\Sym}_n^{\pm}; 2)$
and $\ed(\widetilde{\Alt}_n; 2)$ as $n \longrightarrow \infty$?
Do these numbers grow sublinearly, like $\ed(\Sym_n)$, or exponentially, like $\ed(\Spin_n)$?
\end{question}

Note that for odd primes $p$, 
$\widetilde{\Sym}_n^+$, $\widetilde{\Sym}_n^-$ and $\Sym_n$ have isomorphic Sylow $p$-subgroups,  and thus
$\ed(\widetilde{\Sym}_n^{\pm}; p) = \ed(\Sym_n; p)$; see~Lemma~\ref{lem.Sylow}. Moreover, these numbers are known 
(see, e.g., \cite{reichstein-shukla}*{Remark 6.3} and the references there), and
similarly for $\Alt_n$. For this reason we are only interested in $\ed(\widetilde{\Sym}_n^{\pm}; p)$
and $\ed(\widetilde{\Alt}_n; p)$ when $p = 2$.

In this paper we answer Question~\ref{q1} as follows: $\ed(\widetilde{\Sym}_n^{\pm})$, $\ed(\widetilde{\Alt}_n)$,
$\ed(\widetilde{\Sym}_n^{\pm}; 2)$ grow exponentially if $\Char(k) \neq 2$ and sublinearly
 if $\Char(k) = 2$. This follows from Theorems~\ref{thm.main1} and~\ref{thm.main2} below.

\begin{theorem}\label{thm.main1} Assume that the base field $k$ is of characteristic $\neq 2$ and contains a primitive $8$th root of unity,
and let $n \geqslant 4$ be an integer. Write $n=2^{a_1}+\ldots + 2^{a_s}$, where $a_1>a_2>\ldots > a_s \geqslant 0$ 
and let $\widetilde{\Sym}_n$ be either $\widetilde{\Sym}^+_n$ or $\widetilde{\Sym}^-_n$. Then
	
		\begin{enumerate}
		\item[\rm (a)] $\ed(\widetilde{\Alt}_n) \leqslant \ed(\widetilde{\Sym}_n) \leqslant 2^{\lfloor (n-1)/2 \rfloor}$.
		
		\smallskip
		\item[\rm (b)] $\ed(\widetilde{\Sym}_n; 2) = 2^{\lfloor(n-s)/2 \rfloor}$,
		
		\smallskip
		\item[\rm (c)]	$\ed(\widetilde{\Alt}_n; 2) = 2^{\lfloor(n-s-1)/2 \rfloor}$.
		
		\smallskip
	   \item[\rm (d)] $2^{\lfloor(n-s)/2 \rfloor} \leqslant
		\ed(\widetilde{\Sym}_n)\leqslant \ed(\Sym_n)+ 
		2^{\lfloor(n-s)/2 \rfloor}$ 
		
		\smallskip
		\item[\rm (e)] $2^{\lfloor(n-s-1)/2 \rfloor} \leqslant
		\ed(\widetilde{\Alt}_n) \leqslant \ed(\Alt_n)+ 2^{\lfloor(n-s-1)/2 \rfloor}$.		
		%
 \end{enumerate}
\end{theorem}
	   
For $s = 1$ and $2$, upper and lower bounds of Theorem~\ref{thm.main1} meet, and we obtain the following exact values.

\begin{corollary} \label{cor.main1}  Assume that the base field $k$ contains a primitive $8$th root of unity. 

\smallskip
	     (a) If $n=2^a$, where $a \geqslant 2$, then 
$\ed(\widetilde{\Sym}_n)=\ed(\widetilde{\Sym}_n;2)=\ed(\widetilde{\Alt}_n)=\ed(\widetilde{\Alt}_n;2)=2^{\frac{n-2}{2}}$.

         (b) If $n=2^{a_1}+2^{a_2}$, where $a_1 > a_2 \geqslant 1$, then
         $\ed(\widetilde{\Sym}_n)=\ed(\widetilde{\Sym}_n;2)= 2^{\frac{n-2}{2}}$.
\end{corollary}
         
Note that exact values of $\ed(\Sym_n)$ or $\ed(\Alt_n)$ are only known for $n \leqslant 7$; 
see~\cite[Section 3i]{merkurjev-survey} for a summary. One may thus say that
we know more about $\ed(\widetilde{\Sym}_n)$ and $\ed(\widetilde{\Alt}_n)$ than we do about $\ed(\Sym_n)$ 
and $\ed(\Alt_n)$.
 
\begin{theorem}\label{thm.main2}
       Let $\widetilde{\Sym}_n$ be either $\widetilde{\Sym}^+_n$ or $\widetilde{\Sym}^-_n$. Assume $\Char k=2$. Then
        
        \smallskip
        (a) $\ed(\Sym_n)\leqslant\ed(\widetilde{\Sym}_n)\leqslant\ed(\Sym_n)+1$. 
        
        \smallskip
        (b) $\ed(\Alt_n)\leqslant\ed(\widetilde{\Alt}_n)\leqslant\ed(\Alt_n)+1$.
        
        \smallskip
        (c) $\ed(\widetilde{\Sym}_n; 2) = \ed(\widetilde{\Alt}_n; 2) = 1$.
\end{theorem}

Our proof shows that, more generally, central extensions by $\bbZ/ p \bbZ$ make little difference to the essential dimension of a group over a field of characteristic $p$; see Lemma~\ref{lem.char-p}.

One possible explanation for the slow growth of $\ed(\widetilde{\Sym}_n)$ and $\ed(\widetilde{\Alt}_n)$ in characteristic $2$
is that the connection between $\widetilde{\Sym}^{\pm}_n$ (respectively, $\widetilde{\Alt}_n$) and
$\Pin^{\pm}_n$ (respectively, $\Spin_n$) outlined above breaks down in this setting; see Remark~\ref{rem.char2}.

Some values of $\ed(\Alt_n)$, $\ed(\widetilde{\Alt}_n)$ and $\ed(\widetilde{\Alt}_n; 2)$ over the field
$\bbC$ of complex numbers are shown in Table~\ref{table1} below. Here an entry of the form
$x$-$y$ means that the integer in question lies in the interval $[x, y]$, and 
the exact value is unknown.

\renewcommand{\arraystretch}{1.5}

\begin{table}[ht]
\caption{ \, } 
\centering
 \begin{tabular}{|c|c |c |c |c |c |c |c |c |c |c |c| c | c|} 
 \hline
$n$ & 4 & 5 & 6 & 7 & 8 & 9 & 10 & 11 & 12 & 13 & 14 & 15 & 16 \\ 
 \hline
$\ed_{\bbC}(\Alt_n)$ & 2 & 2 & 3 & 4 & 4-5 & 4-6 & 5-7 & 6-8 & 6-9 & 6-10 & 7-11 & 8-12 & 8-13 \\
 \hline
$\ed_{\bbC}(\widetilde{\Alt}_n;2)$ & 2 &  2 & 2 & 2 & 8 & 8 & 8 & 8 & 16 & 16 & 32 & 32 & 128 \\ 
 \hline 
$\ed_{\bbC}(\widetilde{\Alt}_n)$ & 2 & 2 & 4 &  4 & 8 & 8-14 & 8-15 & 8-16 & 16-25 & 16-26 & 32-43 & 32-44 & 128 \\ 
\hline
\end{tabular}
\label{table1}
\end{table}

As an application of Theorem~\ref{thm.main1}, we will prove the following result in quadratic form theory. 
As usual, we will denote the non-degenerate diagonal 
form $q(x_1, \dots, x_n) = a_1 x_1^2 + \dots + a_n x_n^2$ defined over a field $F$ of characteristic $\neq 2$
by $\langle a_1, \ldots, a_n \rangle$.
Here $a_1, \dots, a_n \in F^*$. We will abbreviate $\langle a, \ldots, a \rangle$ ($m$ times) as $m \langle a \rangle$. 
Recall that the Hasse invariant $w_2(q)$ of 
$q=\langle a_1,\ldots,a_n\rangle$ (otherwise known as the second Stiefel-Whitney class of $q$) is given by 
\[ \text{$w_2(q) = \Sigma_{1 \leqslant i<j \leqslant n} (a_i, a_j)$ in $H^2(F, \bbZ/ 2 \bbZ) = \Br_2(F)$,} \]
where $(a, b)$ is the class of the quaternion algebra 
$F \{ x, y \} /(x^2 = a, \, y^2 = b, \, xy + yx = 0 )$
in $H^2(F,\bbZ / 2 \bbZ)$; see \cite[Section V.3]{lam}. 

Let $E/F$ be an $n$-dimensional \'etale algebra. The trace form $q_{E/F}$ is the $n$-dimensional non-degenerate quadratic form
given by $x \mapsto \Tr_{E/F}(x^2)$. Trace forms have been much studied; for an overview of this research area, see~\cite{bayer}.
An important basic (but still largely open) problem is:
{\em Which $n$-dimensional quadratic forms over $F$ are trace forms?} 

A classification of trace forms of dimension $\leqslant7$, due to J.-P.~Serre, can be found in~\cite[Chapter IX]{serre-ci}.
Note that by a theorem of M.~Epkenhans and K.~Kr\"uskemper~\cite{ek}, for a Hilbertian field $F$ of characteristic $0$
it suffices to consider only field extensions $E/F$. That is, the trace form of any $n$-dimensional \'etale algebra $E/F$
will occur as the trace form of some field extension $E'/F$ of degree $n$. We shall not use this result in the sequel.

Now suppose $F$ contains a primitive $8$th root of unity, and $n = 2^{a_1} + \dots + 2^{a_s}$ is a dyadic expansion of $n$, where
$a_1 > \dots > a_s \geqslant 0$ (as in Theorem~\ref{thm.main1}). Then 
every $n$-dimensional trace form $q$ contains $s \langle 1 \rangle$ as a subform; see, e.g.,~\cite[Proposition 4]{serre-trace}.
This necessary condition for an $n$-dimensional quadratic form to be a trace form is not sufficient;
see Remark~\ref{rem.milnor}. Nevertheless, Theorem~\ref{thm.trace-form} below, tells us that in some ways a general $n$-dimensional trace forms behaves like
a general $n$-dimensional quadratic form that contain $s \langle 1 \rangle$ as a subform.

\begin{theorem} \label{thm.trace-form} Let $k$ be a field containing a primitive $8$th root of unity, $n \geqslant 4$ be an integer,
and $n = 2^{a_1} + \dots + 2^{a_s}$ be the dyadic expansion of $n$, where
$a_1 > \dots > a_s \geqslant 0$. Then

\smallskip
(a) $\displaystyle{{\max_{F, \, q}} \,  \ind (w_2(q))  = \max_{F, \, t} \, \ind(w_2(t)) = 2^{\lfloor (n-s)/2 \rfloor}}$,

\smallskip
(b) $\displaystyle{\max_{F, \, q_1} \,   \ind \, (w_2(q_1))   = \max_{F, \,  t_1} \, \ind(w_2(t_1))  = 2^{\lfloor (n-s -1)/2 \rfloor}}$.  

\smallskip
\noindent
Here the maxima are taken as

\begin{itemize}
\item $F$ ranges over all fields containing $k$, 

\item
$q$ ranges over $n$-dimensional non-degenerate quadratic forms over $F$ containing $s \left< 1 \right>$,

\item
$q_1$ ranges over $n$-dimensional quadratic forms of discriminant $1$ over $F$ containing $s \left< 1 \right>$,

\item
$t$ ranges over $n$-dimensional trace forms over $F$, and

\item
$t_1$ ranges over $n$-dimensional trace forms of discriminant $1$ over $F$,
\end{itemize}
\end{theorem}

Note that if $q = r \oplus s \langle 1 \rangle$, then $q$ and $r$ have the same discriminant and the same Hasse invariant.
Thus in the statement of Theorem~\label{thm.trace}(a) we could replace $\ind(w_2(q))$ by $\ind(w_2(r))$, where $r$ ranges over the $(n-s)$-dimensional non-degenerate quadratic forms over $F$, and similarly in part (b).

The remainder of this paper is structured as follows. Section~\ref{sect.prelim} gives a summary of known results which will 
be needed later on. Theorem~\ref{thm.main1} is proved in Section~\ref{sect.main1-proof}, Theorem~\ref{thm.main2} 
in Section~\ref{sect.main2-proof} and Theorem~\ref{thm.trace-form} in Section~\ref{sect.trace}. 
In Section~\ref{sect.comparison} we compare the essential dimensions 
of $\widetilde{\Sym}_n^+$ and $\widetilde{\Sym}_n^-$, and in Section~\ref{sect.table} we explain the entries 
in Table~\ref{table1}.

\section{Preliminaries}
\label{sect.prelim}

\subsection{Essential dimension}

Recall that the essential dimension of a linear algebraic group $G$ is defined as follows. Let $V$ be a generically free 
linear representation of $G$ and let $X$ be a $G$-variety, i.e., an algebraic variety with an action of $G$. Here
$G$, $V$, $X$ and the $G$-actions on $V$ and $X$ are assumed to be defined over the base field $k$.
We will say that $X$ is generically free if the $G$-action on $X$ is generically free. 
The essential dimension $\ed(G)$ of $G$ is the minimal value of $\dim(X) - \dim(G)$, where $X$ ranges 
over all generically free $G$-varieties admitting a $G$-equivariant dominant rational map $V \dasharrow X$. 
This number depends only on $G$ and $k$ and not on the choice of the generically free representation $V$. 
We will sometimes write $\ed_k(G)$ instead of $\ed(G)$ to emphasize the dependence on $k$.

We will also be interested in the related notion of essential dimension $\ed(G; p)$ of $G$ at a prime integer $p$.
The essential dimension of $G$ at $p$ is defined in the same way as $\ed(G)$, as the minimal value of $\dim(X) - \dim(G)$, where
$X$ is a generically free $G$-variety, except that instead of requiring that $X$ admits a $G$-equivariant dominant rational map $V \dasharrow X$, we only require that it admits a $G$-equivariant dominant
correspondence $V \rightsquigarrow X$ whose degree is prime to $p$. Here by a dominant 
correspondence $V \rightsquigarrow X$ of degree $d$ we mean a diagram of dominant $G$-equivariant rational maps,
\vspace{0.5cm}
  
\centerline{ \xymatrix{ V' \ar@{-->}[d]_{d: 1} \ar@{-->}[drr]  &  &\\
V &  & X.} }
  
\vspace{0.5cm}
  
We will now recall the properties of essential dimension that will be needed in the sequel.
For a detailed discussion of essential dimension and its variants, we refer the reader to
the surveys~\cite{merkurjev-survey} and~\cite{reichstein-icm}.

\begin{lemma} \label{lem.linear} Let $G \hookrightarrow \GL(V)$ be a generically free representation. Then
\[ \ed(G) \leqslant \dim(V) - \dim(G). \]
\end{lemma}

\begin{proof} See \cite[(2.3)]{reichstein-icm} or~\cite[Proposition 2.13]{merkurjev-survey}.
\end{proof}

\begin{lemma} \label{lem.Sylow}
Let $H$ be a closed subgroup of an algebraic group $G$. If the index $[G:H]$ is finite and prime to $p$,
then $\ed(G; p) = \ed(H; p)$.
\end{lemma}

\begin{proof} See~\cite[Lemma 4.1]{meyer-reichstein}.
\end{proof}

\begin{lemma} \label{lem.surjective}
Let $G_1 \to G_2$ be a homomorphism of algebraic groups. If the induced map
\[ H^1(K, G_1) \to H^1(K, G_2) \] is surjective for all field extensions $K$ of $k$, then $\ed(G_1) \geqslant \ed(G_2)$ and
$\ed(G_1; p) \geqslant \ed(G_2; p)$ for every prime $p$.
\end{lemma}

\begin{proof} See~\cite[(1.1)]{reichstein-icm} or~\cite[Proposition 2.3]{merkurjev-survey}.
\end{proof}

\begin{lemma} \label{lem.subgroup} Suppose $H$ is a subgroup of $G$. Then

\smallskip
(a) $\ed(G) \geqslant \ed(H) - \dim(G) + \dim(H)$,

\smallskip
(b) $\ed(G; p) \geqslant \ed(H; p) - \dim(G) + \dim(H)$.
\end{lemma}

\begin{proof} See~\cite[Lemma 2.2]{brv-annals}.
\end{proof}

\subsection{The index of a central extension}
\label{sect.index}
Assume $\Char(k)\neq p.$
Let $G$ be a finite group
and 
\[
\xymatrix{ 1 \ar@{->}[r] &  \bbZ/ p \bbZ \ar@{->}[r] & G \ar@{->}[r] & \overline{G} \ar@{->}[r] &  1} 
\]
be a central exact sequence. This exact sequence gives rise to a connecting morphism 
\[ \delta_K \colon H^1(K, \overline{G}) \to H^2(K, \bbZ/p \bbZ) \,  \]
for every field $K$ containing $k$.
If $K$ contains a primitive $p$th root of unity, then
$H^2(K, \bbZ/p \bbZ)$ can be identified with the $p$-torsion subgroup $\Br_p(K)$ of the Brauer group $\Br(K)$. In particular, we can talk about the index $\ind(\delta_K(\alpha))$ for any $\alpha \in H^2(K, \bbZ/p \bbZ )$. 
Let $\ind(G, \bbZ/p \bbZ)$ denote the maximal value of $\ind(\delta_K(t))$, as $K$ ranges over 
all field extensions of $k$ and $t$ ranges over the elements of $H^1(K, \overline{G})$.

\begin{lemma} \label{lem.index} Assume that the base field $k$ contains a primitive $p$th root of unity. 

\smallskip
(a) If $G_p$ is a Sylow $p$-subgroup of $G$, then 
$\ind(G, \bbZ/p \bbZ) = \ind(G_p, \bbZ/p \bbZ)$.

\smallskip
(b) Suppose the center $Z(G_p)$ is cyclic. Then $\ind(G_p, \bbZ/ p \bbZ) = \ed(G_p) = \ed(G_p; p)$.

\smallskip
(c) $\ed(G) \leqslant \ed(\overline{G}) + \ind(G, \bbZ/p \bbZ)$.
\end{lemma}

\begin{proof}
(a) The diagram 
\[ \xymatrix{ 1 \ar@{->}[r] &  \bbZ/ p \bbZ \ar@{->}[r] & G \ar@{->}[r] & \overline{G} \ar@{->}[r] &  1 \\
 1 \ar@{->}[r] &  \bbZ/ p \bbZ \ar@{->}[r] \ar@{->}[u]^{\simeq} & G_p \ar@{->}[r] \ar@{->}[u] & \overline{G}_p \ar@{->}[r] \ar@{->}[u] &  1,} 
\]
where the rows are central exact sequences and the vertical maps are natural inclusions gives rise to a commutative diagram
\begin{equation} \label{e.connecting} 
\xymatrix{  H^1(K, \overline{G}) \ar@{->}[r]^{\delta_K \; \;} &  H^2(K, \bbZ/ p \bbZ) \\
               H^1(K,  \overline{G_p}) \ar@{->}[r]^{\nu_K \; \;}  \ar@{->}[u]^{i_*} &  H^2(K, \bbZ/ p \bbZ) \ar@{->}[u]^{\simeq} } 
\end{equation}
of Galois cohomology sets for any field $K/k$. Here $\delta$ and $\nu$ denote connecting morphisms.
It is clear from~\eqref{e.connecting} 
that $\ind(G_p, \bbZ/ p \bbZ) \leqslant \ind(G, \bbZ/ p \bbZ)$.
To prove the opposite inequality, choose a field extension $K/k$ and an element $t \in H^1(K, \overline{G})$ such that
$\delta_K(t)$ has the maximal possible index in $H^2(K, \bbZ/ p \bbZ)$; that is, 
\[ \ind(\delta_K(t)) = \ind(G, \bbZ/ p \bbZ). \]
Since $[\overline{G}: \overline{G}_p] = [G: G_p]$ is prime to $p$,
after passing to a suitable finite extension $L/K$ of degree prime to $p$, we may assume that $t_L \in H^1(L, \overline{G})$ is
the image of some $s \in H^1(L, \overline{G_p})$. Here as usual, $t_L$ denotes the image of $t \in H^1(K, \overline{G})$ under the restriction map $H^1(K, \overline{G}) \to H^1(L, \overline{G})$. Since $[L:K]$ is prime to $p$, we have 
\[ \ind(G, \bbZ/ p \bbZ) = \ind(\delta_K(t)) = \ind(\delta_L(t_L)) = \ind(\delta_K(t)_L) = \ind(\nu_L(s)) \leqslant \ind(G_p, \bbZ/ p \bbZ) \, , \]
as desired.

\smallskip
(b) is a variant of a theorem of N.~Karpenko and A.~Merkurjev: the equality  \[ \ind(G_p, \bbZ/ p \bbZ) = \ed(G_p) \]
is a special case of~\cite[Theorem 4.4]{km2},
and the equality $\ed(G_p) = \ed(G_p; p)$ is a part of the statement of~\cite[Theorem 4.1]{km}. 

\smallskip
(c) See \cite[Corollaries 5.8 and 5.12]{merkurjev-survey}; cf.~also \cite[Proposition 2.1]{cernele-reichstein}.
\end{proof}

\subsection{Sylow $2$-subgroups of $\widetilde{\Alt}_n$}

\begin{lemma} \label{lem.small-n}
Let $\widetilde{H}_n$ be a Sylow $2$-subgroup of $\widetilde{\Alt}_n$.

\smallskip
(a) If $n = 4$ or $5$, then $\widetilde{H}_n$ is isomorphic to the quaternion group 
\[ Q_8 = \left< x, y, c \; | \; \;  x^2= y^2 = c, \; \; c^2 = 1, \; \; cx = xc, \; \; cy = yc, \; \; x y = c yx \; \right> . \]

\smallskip
(b) If $n = 6$ or $7$, then $\widetilde{H}_n$ is isomorphic to the generalized quaternion group
\[ Q_{16} = \left< x, y \; | \; \; x^8 = y^4 = 1, \; \; y^2 = x^4, \; \;y x y^{-1} = x^{-1} \; \right> . \]
\end{lemma}

\begin{proof} We will view $\widetilde{\Alt}_n$ (and thus $\widetilde{H}_n$) as a subgroup of $\widetilde{\Sym}_n^+$ and use the 
generators and relations for $\widetilde{\Sym}^{+}_n$ given in the Introduction. 

(a) For $n = 4$ or $5$, we can take
$\widetilde{H}_n$ to be the group of order $8$ generated by $\sigma = (s_1 s_2 s_3)^2$ and $\tau = s_1 s_3$.
These elements project to $(1 \; 3) (2 \; 4)$ and $(1 \; 2) ( 3 \; 4)$ in $\Alt_n$, respectively.
One readily checks that
$\sigma^2 = \tau^2 = z$ and $\sigma \tau = z \tau \sigma$. An isomorphism $Q_8 \to 
\widetilde{H}_n$ can now be defined by 
\[ x \mapsto \sigma, \quad y \mapsto \tau, \quad c \mapsto z .\]

(b) For $n = 6$ or $7$, we can take $\widetilde{H}_n$ to be
the group of order $16$ generated by $\sigma = s_1 s_2 s_3 s_5$ and $\tau = s_1 s_3$.
These elements project to $(1 \; 2 \; 3 \; 4)(5 \; 6)$ and $(1 \; 2) (3 \; 4)$ in $\Alt_n$, respectively.
An isomorphism $Q_{16} \to \widetilde{H}_n$ can now be given by $x \mapsto \sigma$ and $y \mapsto \tau$.
\end{proof}

\begin{proposition} \label{prop.small-n}
Let $n \geqslant 4$ be an integer, $\widetilde{S}_n$ be either $\widetilde{S}^{+}_n$ or $\widetilde{S}^{-}_n$,
and $\widetilde{P}_n$, $\widetilde{H}_n$ be Sylow $2$-subgroups of $\widetilde{S}_n$, $\widetilde{\Alt}_n$, respectively.
Denote the centers of $\widetilde{P}_n$ and $\widetilde{H}_n$ by $Z(\widetilde{P}_n)$ and $Z(\widetilde{H}_n)$, respectively.
Then $Z(\widetilde{P}_n) = Z(\widetilde{H}_n) = \left< z \right>$ is a cyclic group of order $2$. 
\end{proposition}

\begin{proof} By~\cite[Lemma 3.2]{wagner} that $Z(\widetilde{P}_n) = \left< z \right>$ for every $n \geqslant 4$
and $Z(\widetilde{H}_n) = \left< z \right>$ for every $n \geqslant 8$.

It remains to show that $Z(\widetilde{H}_n) = \left< z \right>$ for $4 \leqslant n \leqslant 7$.
Clearly $z \in Z(\widetilde{H}_n)$, so we only need to show that $Z(\widetilde{H}_n)$ is of order $2$.
We will use the description of the groups $\widetilde{H}_n$ from Lemma~\ref{lem.small-n}. If $n = 4$ and $5$,
then $\widetilde{H}_n$ is isomorphic to the quaternion group $Q_8$, and the center of $Q_8$ is clearly of order $2$. 
If $n = 6$ or $7$, then $\widetilde{H}_n \simeq Q_{16}$, and the center
of $Q_{16}$ is readily seen to be the cyclic group $\left< x^4 \right>$ of order $2$.
\end{proof}

\section{Proof of Theorem \ref{thm.main1}} 
\label{sect.main1-proof}

(a) The first inequality follows from the fact that $\widetilde{\Alt}_n$ is contained in $\widetilde{\Sym}_n$; see Lemma~\ref{lem.subgroup}.
For the second inequality, apply Lemma~\ref{lem.linear} to the so-called \textit{basic spin representation} of $\widetilde{\Sym}_n$. This representation is obtained by restricting  a representation of the Clifford algebra $\mathcal{C}_{n-1}$ into $\Mat_{2^{\lfloor \frac{n-1}{2} \rfloor}}(k)$; see \cite[Section 3]{stembridge} for details. (Note that~\cite{stembridge} assumes $k=\mathbb{C}$ but the same morphism works over any field containing square root of $-1$). 

\smallskip
(b) Let $P_n$ be a Sylow $2$-subgroups of $\Sym_n$. The preimage
$\widetilde{P}_n$ of $P_n$ is a Sylow $2$-subgroup 
of $\widetilde{\Sym}_n$.
By Lemma~\ref{lem.Sylow}, $\ed(\widetilde{\Sym}_n; 2) = \ed(\widetilde{P}_n; 2)$.
Moreover, by the Karpenko-Merkurjev theorem~\cite[Theorem 4.1]{km2}, 
\[ \ed(\widetilde{P}_n) = \ed(\widetilde{P}_n; 2) = \dim(V),  \]
where $V$ is a faithful linear representations of $\widetilde{P}_n$ of minimal dimension.

By Proposition~\ref{prop.small-n}, the center of $Z(\widetilde{P}_n) = \left< z \right>$ is of order $2$.
Consequently, a faithful representation $V$ of minimal dimension is
automatically irreducible; see~\cite[Theorem 1.2]{meyer-reichstein-documenta}. 
On the other hand, an irreducible representation $\rho$ of $\widetilde{P}_n$ 
is faithful if and only if $\rho(z) \neq 1$; see, e.g.,~\cite[Lemma 4.1]{wagner}. 
We will now consider several cases.

\smallskip
{\bf Case 1:} Suppose $k = \bbC$ is the field of complex numbers. By~\cite[Lemma 4.2]{wagner}
every irreducible representation 
$\rho \colon \widetilde{P}_n \to \GL(V)$ with $\rho(z) \neq 1$ is of dimension $2^{\lfloor(n-s)/2 \rfloor}$.
This proves part (b) for $k = \mathbb C$. 

\smallskip
{\bf Case 2:} Assume $k$ is a field of characteristic $0$ containing a primitive root of unity
$\zeta_{2^d}$ of degree $2^d$, where $2^d$ is the exponent of $\widetilde{P}_n$.
By a theorem of R.~Brauer~\cite[12.3.24]{serre-lr}, every irreducible complex representation of $\widetilde{P}_n$ 
is, in fact, defined over $k$. Thus the dimension of the minimal faithful irreducible representation over $k$ is the
same as over $\bbC$, i.e., $2^{\lfloor(n-s)/2 \rfloor}$, and part (b) holds over $k$.

\smallskip
{\bf Case 3:} Now suppose that $k$ is a field of characteristic $0$ containing $\zeta_8$ (but possibly not $\zeta_{2^d}$).
Set $l = k(\zeta_{2^d})$. Then
\[ \ed_k(\widetilde{P}_n) \geqslant \ed_l(\widetilde{P}_n) = 2^{\lfloor(n-s)/2 \rfloor} . \]
To prove the opposite inequality, let $V$ be a faithful irreducible representation of $\widetilde{P}_n$
dimension $2^{\lfloor(n-s)/2 \rfloor}$
defined over $\mathbb Q(\zeta_{2^d})$. Such a representation exists by Case 2. We claim that $V$ is, in fact, defined over $\bbQ(\zeta_8)$. In particular, $V$ is defined over $k$ and
thus \[ \ed_k(\widetilde{P}_n) \leqslant 2^{\lfloor(n-s)/2 \rfloor} , \]
as desired. We will prove the claim in two steps. 

First we will show that the character $\chi \colon \widetilde{P}_n \to \bbQ(\zeta_{2^d})$ of $V$ takes all of its values in $\bbQ(\zeta_8)$.
By \cite[Lemma 4.2]{wagner}, there are either one or two faithful irreducible characters of $\widetilde{P}_n$ of dimension $2^{\lfloor(n-s)/2 \rfloor}$.
The Galois group $G = \Gal(\bbQ(\zeta_{2^d})/\bbQ) \simeq \bbZ/ 2\bbZ \times \bbZ/2^{d-2} \bbZ$ 
acts on this set of characters. Thus for any $\sigma \in \widetilde{P}_n$,
the $G$-orbit of $\chi(\sigma)$ has either one or two elements. Consequently, $[\bbQ(\chi(\sigma)) : \bbQ] = 1$ or $2$. Note that $G$ has exactly three subgroups of index $2$.
Under the Galois correspondence these subgroups correspond to the subfields $\bbQ(\sqrt{-1})$, $\bbQ(\sqrt{2})$ and $\bbQ(\sqrt{-2})$ of $\bbQ(\zeta_{2^d})$. Thus $\chi(\sigma)$ lies in one of these three fields; in particular,
$\chi(\sigma) \in \bbQ(\sqrt{-1}, \sqrt{2}) = \bbQ(\zeta_8)$ for every $\sigma \in G$. In other words, 
$\chi$ takes all of its values in $\bbQ(\zeta_8)$, as desired.

Now observe that since $\widetilde{P}_n$ is a $2$-group, the Schur index of $\chi$ over $\bbQ(\zeta_8)$ is $1$; see~\cite[Corollary 9.6]{yamada}. Since the character $\chi$ of $V$ is defined over $\bbQ(\zeta_8)$ and the 
Schur index of $\chi$ is $1$, we conclude that $V$ itself is defined over $k$. 
This completes the proof of part (b) in Case 3 (i.e., in characteristic $0$).

\smallskip
{\bf Case 4:} Now assume that $k$ is a perfect field of characteristic $p > 2$ containing $\zeta_8$. 
Let $A = W(k)$ be the Witt ring of $k$. Recall that $A$ is a complete discrete valuation ring 
of characteristic zero, whose residue field is $k$. By Hensel's Lemma, $\zeta_8$ lifts to a primitive $8$th root of unity in $A$.
Denote the fraction field of $A$ by $K$ and the maximal ideal by $M$. 
Since $\widetilde{P}_n$ is a $2$-group and $\Char(k) = p$ is an odd 
prime, every $d$-dimensional $k[\widetilde{P}_n]$-module $W$ lifts to a unique $A[\widetilde{P}_n]$-module 
$W_A$, which is free of rank $n$ over $A$.  
Moreover, the lifting operation $V \mapsto V_K := V_A \otimes K$
and the ``reduction mod $M$" operation give rise to mutually inverse bijections between 
$k[\widetilde{P}_n]$-modules and $K[\widetilde{P}_n]$-modules; see~\cite[Section 15.5]{serre-lr}.
These bijections preserve dimension and faithfulness of modules. Since $K$ is a field of characteristic $0$
containing a primitive $8$th root of unity, Case 3 tells us that the minimal dimension of a faithful $K[\widetilde{P}_n]$-module is
$2^{\lfloor(n-s)/2 \rfloor}$. Hence, the minimal dimension of a faithful $k[\widetilde{P}_n]$-module is also
$2^{\lfloor(n-s)/2 \rfloor}$. This proves part (b) in Case 4. 

\smallskip
{\bf Case 5:} Now assume that $k$ is an arbitrary field of characteristic $p > 2$ containing $\zeta_8$.
Denote the prime field of $k$ by $\bbF_p$.
Then $k$ can be sandwiched between two perfect fields, $k_1 \subset k \subset k_2$, where
$k_1 = \mathbb{F}_p(\zeta_8)$ is a finite field, and $k_2$ is the algebraic closure of $k$.
Then \[ \ed_{k_1}(\widetilde{P}_n) \geqslant \ed_k(\widetilde{P}_n) \geqslant \ed_{k_2}(\widetilde{P}_n). \]
By Case 4,
$\ed_{k_1}(\widetilde{P}_n) = \ed_{k_2}(\widetilde{P}_n) = 2^{\lfloor(n-s)/2 \rfloor}$. 
We conclude that $\ed_{k}(\widetilde{P}_n) = 2^{\lfloor(n-s)/2 \rfloor}$.
The proof of part (b) is now complete.

\smallskip
(c) Let  $H_n$ be the Sylow $2$-subgroups of $\Alt_n$. Its preimage
$\widetilde{H}_n$ is a Sylow $2$-subgroup of $\widetilde{\Alt}_n$. 
By the Karpenko-Merkurjev theorem~\cite[Theorem 4.1]{km2}, 
$\ed(\widetilde{H}_n) = \ed(\widetilde{H}_n; 2) = \dim(W)$,
where $W$ is a faithful linear representation of $\widetilde{H}_n$ of minimal dimension.

The rest of the argument in part (b) goes through with only minor changes. Once 
again, by Proposition~\ref{prop.small-n}, the center of $Z(\widetilde{H}_n) = 
\left< z \right>$ is of order $2$. Thus $W$ is irreducible. Moreover,
an irreducible representation $\rho$ of $\widetilde{H}_n$ 
is faithful if and only if $\rho(z) \neq 1$.

If $k = \mathbb C$ is the field of complex numbers, it is shown in~\cite[Lemma 4.3]{wagner} that
every irreducible representation $\rho$ of $\widetilde{H}_n$ with $\rho(z) \neq 1$ is of dimension
$2^{\lfloor(n-s -1)/2 \rfloor}$. This proves part (c) for $k = \bbC$.
Moreover, depending on the parity of $n - s$, there are either one or two such representations. 
For other base fields $k$ containing a primitive $8$th root of unity (in Cases 2-5) the arguments 
of part (b) go through unchanged.

\smallskip
(d) Since $\ed(\widetilde{\Sym}_n) \geqslant \ed(\widetilde{\Sym}_n; 2)$, 
the lower bound of part (d) follows immediately from part (b).
To prove the upper bound, we apply Lemma~\ref{lem.index} to the exact sequence~\eqref{e.extension}.
Here $p = 2$,  and $\bbZ/ 2 \bbZ = \left< z \right>$ is the center of $G_p = \widetilde{P}_n$ by Proposition~\ref{prop.small-n}.

\smallskip
(e) The lower bound follows from part (c) and the inequality $\ed(\widetilde{\Alt}_n) \geqslant \ed(\widetilde{\Alt}_n; 2)$. 
The upper bound is obtained by applying Lemma~\ref{lem.index} to the exact sequence 
\[ \xymatrix{ 1 \ar@{->}[r] &  \left<  z \right> \ar@{->}[r] & \widetilde{\Alt}_n \ar@{->}[r] &
\Alt_n \ar@{->}[r] &  1,} \]
in the same way as in part (d).

\begin{remark} \label{rem.roots-of-unity} If $n - s$ is even, then $\widetilde{\Sym}_n$ has only one faithful irreducible complex representation of dimension 
$2^{\lfloor(n-s)/2 \rfloor}$; see~\cite[Lemma 4.2]{wagner}. In this case we can relax the assumption on $k$ in part (b) a little bit: 
our proof  goes through for any base field $k$ containing 
$\zeta_4 = \sqrt{-1}$. (Similarly for part (c) in the case where $n - s$ is odd; see~\cite[Lemma 4.3]{wagner}.)

However, in general part (b) fails if we do not assume that $\zeta_8 \in k$. For example, in the case, where $s = 1$ (i.e., $n \geqslant 4$ is a power of $2$),
the inequality $\ed(\widetilde{\Sym_n}; 2) \leqslant 2^{\lfloor(n-s)/2 \rfloor} = 2^{(n-2)/2} $ is equivalent to the existence of a faithful irreducible 
representation $V$ of $\widetilde{P}_n$ of degree $2^{(n-2)/2}$ defined over $k$. There are two such representations, and \cite[Theorem 8.7]{hoff}
shows that some of their character values are not contained in $\bbQ(\zeta_4)$.
\end{remark}

\section{Proof of Theorem \ref{thm.main2}}
\label{sect.main2-proof}

Part (c) follows directly from~\cite[Theorem 1]{reichstein-vistoli-p}, which says that if
$G$ be a finite group and $k$ is a field of characteristic $p > 0$, then 
\[ \ed_k(G; p) = \begin{cases} \text{$1$, 
if the order of $G$ is divisible by $p$, and} \\
\text{$\ed_k(G; p) = 0$, otherwise.}
\end{cases} \]
In particular, $\ed_k(\widetilde{\Sym}_n; 2) = \ed_k(\widetilde{\Alt}_n; 2) = 1$ for any field $k$ of characteristic $2$. 

\smallskip
Parts (a) and (b) are consequences of the following lemma. In the case, where $G$ is a finite $p$-group, this lemma is due to A.~Ledet; see~\cite[Theorem 1]{ledet-p}.

\begin{lemma} \label{lem.char-p} Let $k$ be a field of characteristic $p$, $G$ be a linear
algebraic group defined over $k$ and
\[
\xymatrix{ 1 \ar@{->}[r] &  \bbZ/ p \bbZ \ar@{->}[r] & G \ar@{->}[r] & \overline{G} \ar@{->}[r] &  1} 
\]
be a central exact sequence. Then $\ed_k(\overline{G}) \leqslant \ed_k(G) \leqslant \ed_k(\overline{G}) + 1$.
\end{lemma}

\begin{proof}
(a) Consider the induced exact sequence
\[ \xymatrix{  H^1(K, G) \ar@{->}[r] & H^1(K, \overline{G}) \ar@{->}[r]^{\delta_K \; \; } &  H^2(K, \bbZ/ p \bbZ) } \]
in Galois cohomology (or flat cohomology, if $\overline{G}$ is not smooth), where $\delta_K$ denotes the boundary map.
Here $K/k$ is an arbitrary field extension $K/k$. Since $K$ is a field of characteristic $p$, its cohomological $p$-dimension is $\leqslant 1$ 
and thus $H^2(K, \bbZ/ p \bbZ) = 1$; see~\cite[Proposition II.2.2.3]{cg}. In other words, 
the map $H^1(K, G) \to  H^1(K, \overline{G})$ is surjective for any field $K$ containing $k$. By Lemma~\ref{lem.surjective}, this implies 
\[ \ed_k(G) \geqslant \ed_k(\overline{G}). \] 
On the other hand, since $\mathbb{Z}/p \mathbb{Z}$ is unipotent in characteristic $p$,~\cite[Lemma 3.4]{tossici-vistoli} tells us that
\[ \ed_k(G)\leqslant\ed_k(\overline{G})+\ed_k(\mathbb{Z}/p \mathbb{Z}) = \ed_k(\overline{G}) + 1 ; \] 
see also~\cite[Corollary 3.5]{lotscher}. 
\end{proof}

\begin{remark} \label{rem.char2}
Note that in characteristic $2$ the group $\widetilde{\Alt}_n$ is no longer isomorphic to the preimage 
of $\Alt_n \subset \SO_n$ in $\Spin_n$.  The scheme-theoretic preimage of $\Alt_n$ in $\Spin_n$ 
is an extension of a constant group scheme $\Alt_n$ by an infinitesimal group scheme $\mu_2$. 
Any such extension is split over a perfect base field; see, e.g.,~\cite[Proposition 15.22]{milne}.
Thus, over a perfect field $k$, the preimage of $\Alt_n \subset \SO_n$ in $\Spin_n$ is the direct product $\Alt_n \times \mu_2$.
\end{remark}

\begin{remark} \label{rem2.char2}
As we mentioned in the Introduction, the exact values of $\ed(\Sym_n)$ and
$\ed(\Alt_n)$ in characteristic $0$ are not known for any $n \geqslant 8$ . 
In characteristic $2$, even less is known. 
The upper bound, 
\[ \text{$\ed_k(\Alt_n) \leqslant \ed_k(\Sym_n) \leqslant n-3$ for any $n \geqslant 5$,} \]
is valid over an arbitrary field $k$. 
It is also known that if $G$ is a finite group and $G$ does not have a non-trivial normal $2$-subgroup, then
$\ed_k(G) \leqslant \ed_{\mathbb C}(G)$ for any field $k$
of characteristic $2$ containing the algebraic closure of $\bbF_2$; see~\cite[Corollary 3.4(b)]{brv-mixed}.
In particular, this applies to $G = \Sym_n$ or $\Alt_n$ for any $n \geqslant 5$.

In characteristic $0$, $\ed(\Sym_n) \geqslant \lfloor n/2 \rfloor$ for any $n \geqslant 1$ and
$\ed(\Sym_n) \geqslant \lfloor (n+1)/2 \rfloor$ for any $n \geqslant 7$. It is not known if these inequalities remain 
true in characteristic $2$. On the other hand, since 
$\Alt_n$ contains $(\bbZ / 3 \bbZ)^r$, where $r = \lfloor n/3 \rfloor$, it is easy to see that
the weaker inequality
\[ \ed(\Sym_n) \geqslant \ed(\Alt_n) \geqslant \lfloor n/3 \rfloor \]
remains valid in characteristic $2$. For general $n$, this is the best lower bound we know.
\end{remark}

\begin{example} \label{ex.char2} 
Assume that the base field $k$ is infinite of characteristic $2$. We claim that
\[\ed_k(\Sym_4) = 2. \]
Let $P_4$ be a Sylow $2$-subgroup of $\Sym_4$. Recall that $P_4$ is isomorphic to the dihedral group of order $8$.
By~\cite[Proposition 7]{ledet-ed1}, $\ed_k(P_4) \geqslant 2$ and thus 
$\ed_k(\Sym_4) \geqslant 2$. To prove the opposite inequality, consider the faithful
$3$-dimensional representation of $\Sym_4$ given by
\[ V = \{ (x_1, x_2, x_3, x_4) \, | \, x_1 + x_2 + x_3 + x_4 = 0 \} . \]
Here $\Sym_4$ acts on $V$ by permuting $x_1, \ldots, x_4$. The natural compression
$V \dasharrow \bbP(V)$ shows that $\ed_k(\Sym_4) \leqslant 2$. This proves the claim.
By Theorem~\ref{thm.main2} we conclude that
\[ 2 \leqslant \ed_k(\widetilde{\Sym}_4) \leqslant 3. \]
We do not know whether $\ed_k(\widetilde{\Sym}_4) = 2$ or $3$. Note however, 
that by a conjecture of Ledet~\cite[p.~4]{ledet-p}, $\ed_k(\bbZ/2^n \bbZ) = n$ for every integer $n \geqslant 1$.
Since $\widetilde{\Sym}_4$ contains an element of order $8$ (the preimage of a $4$-cycle in $\Sym_4$),
Ledet's conjecture implies that $\ed_k(\widetilde{\Sym}_4) = 3$. Note also that by Corollary~\ref{cor.main1},
$\ed_l(\widetilde{\Sym}_4) = 2$ for any base field $l$ of characteristic $\neq 2$ containing 
a primitive $8$th root of unity.
\end{example}

\section{Proof of Theorem~\ref{thm.trace-form}}
\label{sect.trace}

Let $q=\langle a_1,\ldots,a_n\rangle$ be a non-degenerate $n$-dimensional quadratic form over a field $F$.
Recall from the Introduction that the Hasse invariant $w_2(q)$ is given by 
\[ \text{$w_2(q) = \Sigma_{1\leqslant i<j \leqslant n} (a_i, a_j)$ in $H^2(F, \bbZ/ 2 \bbZ) = \Br_2(F)$,} \]
where $(a, b) = (a)\cup (b)$ is the class of the quaternion algebra 
\[F \{ x, y \} /(x^2 = a, \, y^2 = b, \, xy + yx = 0 ). \]
It is immediate from this definition that
\begin{equation} \label{e.hasse}
w_2(\left< 1 \right> \oplus q) = w_2(q).
\end{equation}

Our proof of Theorem~\ref{thm.trace-form} will be based on the following elementary lemma.

\begin{lemma} \label{lem.hasse}
Let $F$ be a field of characteristic $\neq 2$ containing a primitive $4$th root of unity.
Let $q$ be an $n$-dimensional non-degenerate quadratic form over $F$. Then

\smallskip
(a) $\ind (w_2(q)) \leqslant 2^{\lfloor n/2 \rfloor}$. 

\smallskip
(b) If $q$ is of discriminant $1$ over $F$, then $\ind (w_2(q)) \leqslant 2^{\lfloor (n-1)/2 \rfloor}$.
\end{lemma}

\begin{proof}
Let $q=\langle a_1,a_2,\ldots,a_{n-1},a_n\rangle$ for some $a_1, \dots, a_n \in F^*$.

\smallskip
(a) We will consider the cases where $n$ is odd and even separately.
If $n = 2m$ even, then $S = F\left(\sqrt{-\frac{a_1}{a_2}},\ldots,\sqrt{-\frac{a_{2m-1}}{a_{2m}}}\right)$ splits $q$. That is
$q_S \simeq n \langle 1 \rangle$. Hence, $S$ also splits $w_2(q)$. Since the index of an element $\alpha \in H^2(F, \bbZ/ 2\bbZ)$
is the minimal degree $[K:F]$ of a splitting field $K/F$, we conclude that $\ind (w_2(q)) \leqslant [S:F] = 2^{m}$, as desired. 

Now suppose that $n = 2m + 1$ is odd. Set $S = F \left(\sqrt{-\frac{a_1}{a_2}},\ldots,\sqrt{-\frac{a_{2m-1}}{a_{2m}}}\right)$, as before. Over $S$, 
\[ q_S \simeq 2m \langle 1 \rangle  \oplus \langle a_{2m + 1} \rangle . \]
It now follows from~\eqref{e.hasse} that $w_2(q_S) = 0$ in $H^2(S, \bbZ/ 2\bbZ)$.
Hence, $w_2(q)$ splits over $S$ and consequently, $\ind(w_2(q))\leqslant [S: F] = 2^m = 2^{\frac{n-1}{2}}$, as desired. 

\smallskip
(b) Since $q= \langle a_1, \ldots, a_n \rangle$ has discriminant $1$, we may assume without loss of generality that
$a_1 \cdot \ldots \cdot a_n = 1$ in $F$. Let $r = \langle a_2, \dots, a_{n} \rangle$.
Since $K$ contains $\sqrt{-1}$, the quaternion algebra $(a_1, a_1)$ is split. Thus 
\[ \begin{array}{r} w_2(q)= (a_1, a_1) + w_2(q) =  
(a_1, a_1 \cdot \ldots \cdot a_n) + w_2(r)\\
=(a_1, 1) + w_2(r) = w_2(r) . \end{array} \]
By part (a), $\ind (w_2(r)) \leqslant 2^{\lfloor (n-1)/2 \rfloor}$, and part (b) follows.
\end{proof}

We are now ready to proceed with the proof of Theorem~\ref{thm.trace-form}.

\smallskip
(a) As we pointed out in the Introduction, every $n$-dimensional trace form contains $s \langle 1 \rangle$ as a subform; 
see~\cite[Proposition 4]{serre-trace}. Thus
\begin{equation} \label{e.trace1}
\max_{F, \, t} \, \ind(w_2(t))  \leqslant  \max_{F, \, q} \, \ind (w_2(q)).
\end{equation}
On the other hand, 
by our assumption, $q = s\left< 1 \right> \oplus r$, where $r$ is a form of dimension $n - s$.
By~\eqref{e.hasse}, $w_2(q) = w_2(r)$ and by Lemma~\ref{lem.hasse}(a),
$\ind (w_2(r)) \leqslant 2^{\lfloor (n-s)/2 \rfloor}$.
Thus 
\begin{equation} \label{e.trace2} \max_{F, \, q} \, \ind ( w_2(q)) \leqslant 2^{\lfloor (n-s)/2 \rfloor}.
\end{equation}
In view of the inequalities~\eqref{e.trace1} and \eqref{e.trace2}, it suffices to show that
\begin{equation} \label{e.trace3} 
\max_{F, \, t} \, \ind(w_2(t))  =  2^{\lfloor (n-s)/2 \rfloor}.
\end{equation} 
Recall that elements of $H^1(F, \Sym_n)$ are in a natural bijective correspondence with isomorphism classes of 
$n$-dimensional \'etale algebras $E/F$. Denote the class of $E/F$ by $[E/F] \in H^1(F, \Sym_n)$ and the trace form of $E/F$ by $t$.
By~\cite[Th\'eor\`eme 1]{serre-trace}, 
\[ \delta ([E/F]) = w_2(t); \]
cf.~also~\cite[Section 9.2]{serre-topics}. Thus the largest value of the index of $w_2(t)$,
as $F$ ranges over all field extensions of $k$ and $E/F$ ranges over all $n$-dimensional \'etale $F$-algebras,
is precisely the integer $\ind(\widetilde{\Sym}_n, \bbZ/2 \bbZ)$ defined in Section~\ref{sect.index}.
Let $\widetilde{P}_n$ be a Sylow subgroup of $\widetilde{\Sym}_n$. By Proposition~\ref{prop.small-n}, the center
$Z(\widetilde{P}_n)$ is cyclic. Thus by Lemma~\ref{lem.index},
\[ \ind(\widetilde{\Sym}_n; \bbZ/ 2\bbZ) = \ed(\widetilde{P}_n) = \ed(\widetilde{P}_n; 2) . \]
On the other hand, $\ed(\widetilde{P}_n; 2) = \ed(\widetilde{\Sym}_n; 2)$ by~Lemma~\ref{lem.Sylow} and
$\ed(\widetilde{\Sym}_n; 2) = 2^{\lfloor (n-s)/2 \rfloor}$
by Theorem~\ref{thm.main1}(b). This completes the proof of~\eqref{e.trace3} and thus of part (a) of Theorem~\ref{thm.trace-form}.

\smallskip
The proof of part (b) is similar. Once again, since $t_1$ contains $s \langle 1 \rangle$ as a subform,
$\displaystyle{\max_{F, \, t_1} \ind( w_2(t_1)) \leqslant \max_{F, \, q_1} \, \ind ( w_2(q_1))}$. On the other hand, 
$\displaystyle{\max_{F, \, q_1} \, \ind ( w_2(q_1))  \leqslant 2^{\lfloor (n-s - 1)/2 \rfloor}}$ by Lemma~\ref{lem.hasse}(b).  It thus remains to show that
\begin{equation} \label{e.trace4}
\max \{ \ind ( w_2(t_1)) \} \leqslant 2^{\lfloor (n-s- 1)/2 \rfloor}.
\end{equation}
Consider the diagram
\[ \xymatrix{ 1 \ar@{->}[r] &  \bbZ/ 2 \bbZ \ar@{->}[r] \ar@{=}[d] & \widetilde{\Sym}_n \ar@{->}[r] &
\Sym_n \ar@{->}[r] &  1 \\
1 \ar@{->}[r] &  \bbZ/ 2 \bbZ \ar@{->}[r] & \widetilde{\Alt}_n \ar@{->}[r] \ar@{^{(}->}[u] &  
\Alt_n \ar@{->}[r] \ar@{^{(}->}[u]_i &  1} 
\]
where $\widetilde{\Sym}_n$ can be either $\widetilde{\Sym}_n^+$ or $\widetilde{\Sym}_n^-$. Since the rows are exact, the connecting morphisms
fit into a commutative diagram
\[ \xymatrix{ H^1(F, \Sym_n) \ar@{->}[r]^{w_2} &  H^2(F, \bbZ/ 2 \bbZ) \\
H^1(F, \Alt_n) \ar@{->}[r]^{\partial_F} \ar@{->}[u]^{i_*} &  H^2(F, \bbZ/ 2 \bbZ) \ar@{=}[u].} \]
Once again, elements of $H^1(F, \Sym_n)$ are in a natural bijective correspondence with $n$-dimensional etale algebras $E/F$.
The image of the vertical map $i_* \colon H^1(F, \Alt_n) \to  H^1(F, \Sym_n)$ is readily seen to consist of etale algebras $E/F$ of discriminant $1$.
Consequently,
\[ \max \{ \ind (w_2(t_1)) \} =\ind(\widetilde{\Alt}_n, \bbZ/2 \bbZ) . \]
Let $\widetilde{H}_n$ be a Sylow subgroup of $\widetilde{\Alt}_n$. By Proposition~\ref{prop.small-n}, the center
$Z(\widetilde{H}_n)$ is cyclic. Thus by Lemma~\ref{lem.index},
\[ \ind(\widetilde{\Alt}_n; \bbZ/ 2\bbZ) = \ed(\widetilde{H}_n) = \ed(\widetilde{H}_n; 2) . \]
On the other hand, $\ed(\widetilde{H}_n; 2) = \ed(\widetilde{\Alt}_n; 2)$ by~Lemma~\ref{lem.Sylow} and
$\ed(\widetilde{\Alt}_n; 2) = 2^{\lfloor (n-s -1)/2 \rfloor}$
by Theorem~\ref{thm.main1}(c). This completes the proof of~\eqref{e.trace4}.
\qed

\begin{remark} One can use the inequalities~\eqref{e.trace1} and~\eqref{e.trace2} to give an alternative proof of the upper bound 
$\ed(\widetilde{\Sym}_n) \leqslant 2^{\lfloor (n-s)/2 \rfloor}$ of Theorem~\ref{thm.main1}(b). Similarly for the upper bound
$\ed(\widetilde{\Alt}_n) \leqslant 2^{\lfloor (n-s -1)/2 \rfloor}$ in the proof of Theorem~\ref{thm.main1}(c). On the other hand, we do not know how to prove 
the lower bounds $\ed(\widetilde{\Sym}_n) \geqslant 2^{\lfloor (n-s)/2 \rfloor}$ and
$\ed(\widetilde{\Alt}_n) \geqslant 2^{\lfloor (n-s -1)/2 \rfloor}$ entirely within the framework
of quadratic form theory, without the representation-theoretic input from~\cite{wagner}.
\end{remark} 

\begin{remark} \label{rem.sufficient} Let $F$ be a field of characteristic $\neq 2$ containing a primitive $8$th root of unity,
and $q$ be an $n$-dimensional non-degenerate quadratic form over $F$. 
As we pointed out in the Introduction, a necessary condition for $q$ to be a trace form is that it should contain $s \langle 1 \rangle$ as a subform.
Theorem~\ref{thm.trace-form} suggests that this condition might be sufficient.
The following example shows that, in fact, this condition is not sufficient. In this example, $n = 4 = 2^2$ and thus $s = 1$.
Let $k$ be an arbitrary base field of characteristic $\neq 2$, $a$, $b$, $c$ be independent variables, $F = k(a, b, c)$, and
$q = \langle 1, a, b, c \rangle$ be a $4$-dimensional non-singular quadratic form over $F$. Clearly $q$ contains $s \langle 1 \rangle =
\langle 1 \rangle$ as a subform. On the other hand, $q$ is not isomorphic to the trace form 
$t$ of any $4$-dimensional etale algebra $E/F$. Indeed, $\ed_k(t) \leqslant \ed_k(E/F) \leqslant \ed_k(\Sym_4) = 2$ (see~\cite[Theorem 
6.5(a)]{bur}), whereas $\ed_k(q) = 3$ (see~\cite[Proposition 6]{cs}).  
\end{remark}

\begin{remark} \label{rem.milnor} Recall that by a theorem of Merkurjev~\cite{merkurjev-milnor},
$w_2$ gives rise to an isomorphism between $I^2(K)/I^3(K)$ and $H^2(K, \bbZ/ 2\bbZ)$; cf.~\cite[p.~115]{lam}.
This is a special case of Milnor's conjecture, which asserts the existence of an isomorphism
\[ e_r \colon I^r(F)/I^{r+1}(F) \to H^r(F, \bbZ/ 2\bbZ) \] for any $r \geqslant 0$, with the property that $e_r$
takes the $r$-fold Pfister form $\langle 1, a_1 \rangle \otimes \dots \otimes \langle 1, a_r \rangle$
to the cup product $(a_1) \cup (a_2) \cup \dots \cup (a_r)$; see~\cite[p.~33]{pfister}. Milnor's conjecture has been
proved by V.~Voevodsky; see \cite{kahn} for an overview.
It is natural to ask if the following variant of Theorem~\ref{thm.trace-form} remains valid for every $r \geqslant 1$.

\smallskip
Let $k$ be a base field containing a primitive $8$th root of unity, 
and $n = 2^{a_1} + \dots + 2^{a_s}$ be an even positive integer, where $a_1 > \dots > a_s \geqslant 1$. Is it true that
\begin{equation} \label{e.milnor}
\max_{F, \, q} \,  \ind (e_r(q))  = \max_{F, \, t} \, \ind(e_r(t)) \, ?
\end{equation}
\noindent
Here the maximum is taken as $F$ ranges over all fields containing $k$, 
$q$ ranges over all $n$-dimensional forms in $I^r(F)$ containing $s \left< 1 \right>$\footnote{If $s$ is even, and $q = r \oplus s \langle 1 \rangle$, then $q$ and $r$ are Witt equivalent over $F$. Thus $\max_{F, \, q} \,  \ind (e_r(q))$ can be replaced
by $\max_{F, \, r} \,  \ind (e_r(r))$, as $r$ ranges over all $(n-s)$-dimensional forms in $I^r(F)$. The same is true if $s$ is odd: here $r$ ranges over the $(n-s)$-dimensional forms in $I^r(F)$ such that $r \oplus \langle 1 \rangle$ is in $I^r(F)$.} and
$t$ ranges over all $n$-dimensional trace forms in $I^r(F)$. The index of a class $\alpha \in H^r(F, \bbZ/ 2 \bbZ)$ is
the greatest common divisor of the degrees $[E:F]$, where $E/F$ ranges over splitting fields for $\alpha$ with $[E:F] < \infty$.

For $r = 1$, it is easy to see that~\eqref{e.milnor} holds.
In this case the Milnor map $e_1 \colon I^1(F)/I^2(F) \to H^1(F, \bbZ/ 2 \bbZ)$ is the discriminant, 
the index of an element of $\alpha = H^1(F, \bbZ /2 \bbZ) = F^*/(F^*)^2$ is $1$ or $2$, depending on whether 
$\alpha = 0$ or not, and the question boils down to the existence of an $n$-dimensional \'etale algebra $E/F$ with non-trivial discriminant.
In the case where $r = 2$, the equality~\eqref{e.milnor} is given by Theorem~\ref{thm.trace-form}(b) (where $n$ is taken to be even).
\end{remark}

\section{Comparison of essential dimensions of $\widetilde{\Sym}_n^{+}$ and $\widetilde{\Sym}_n^{-}$}
\label{sect.comparison}

We believe that $\widetilde{\Sym}_n^{+}$ and $\widetilde{\Sym}_n^{-}$ should have the same essential dimension, 
but are only able to establish the following slightly weaker assertion.

\begin{proposition} \label{lem.pm} Let $k$ be a field of characteristic $\neq 2$ containing $\sqrt{-1}$. Then
\[ |\ed_k(\widetilde{\Sym}_n^+) - \ed_k(\widetilde{\Sym}_n^-)| \leqslant 1. \]
\end{proposition}

\begin{proof}
Let $V$ be the spin representation of $\widetilde{\Sym}^+_n$, $\Sym_n \to \PGL(V)$ be the associated projective  representation of $\Sym_n$, and $\Gamma \subset \GL(V)$ be the preimage of $G$ under the natural projection $\pi \colon \GL(V) \to \PGL(V)$.
Note that $\Gamma$ is a $1$-dimensional group, and $\widetilde{\Sym}_n^{\pm}$ are finite subgroups of $\Gamma$. 
By Lemma~\ref{lem.subgroup}(a),
\begin{equation} \label{e.e1}
 \ed_k(\Gamma) \geqslant \ed_k(\widetilde{\Sym}_n) - 1, 
 \end{equation}
where $\widetilde{\Sym}_n$ denotes $\widetilde{\Sym}_n^+$ or $\widetilde{\Sym}_n^-$.
On the other hand, since $\Gamma$ is generated by $\widetilde{\Sym}_n$ and $\bbG_m$, and $\bbG_m$ is central in $\Gamma$, we see
that $\widetilde{\Sym_n}$ is normal in $\Gamma$. The exact sequence
\[ \xymatrix{1 \ar@{->}[r] & \widetilde{\Sym}_n \ar@{->}[r] & \Gamma \ar@{->}[r]^{\pi \; } & \bbG_m  \ar@{->}[r] & 1} \]
induces an exact sequence 
\[ \xymatrix{H^1(K, \widetilde{\Sym}_n) \ar@{->}[r] & H^1(K, \Gamma) \ar@{->}[r]^{\pi \quad \; } & H^1(K, \bbG_m) = 1} \]
of Galois cohomology sets. Here $K$ is an arbitrary field containing $k$, and $H^1(K, \bbG_m) = 1$ by Hilbert's Theorem 90. Thus
the map
$H^1(K, \widetilde{\Sym}_n) \to H^1(K, \Gamma)$ is surjective for every $K$. By Lemma~\ref{lem.surjective},
\begin{equation} \label{e.e2} 
\ed_k(\widetilde{\Sym}_n) \geqslant \ed_k(\Gamma). 
\end{equation}
Combining the inequalities~\eqref{e.e1} and~\eqref{e.e2}, we see that each 
of the integers $\ed_k(\widetilde{\Sym}_n^+)$ and $\ed_k(\widetilde{\Sym}_n^-)$ equals either
$\ed_k(\Gamma)$ or $\ed_k(\Gamma) + 1$. Hence, $\ed_k(\widetilde{\Sym}_n^+)$ and $\ed_k(\widetilde{\Sym}_n^-)$ differ by at most $1$, as claimed.
\end{proof}

\begin{remark} The inequality $|\ed_k(\widetilde{\Sym}_n^+) - \ed_k(\widetilde{\Sym}_n^-)| \leqslant 1$ of Lemma~\ref{lem.pm}
remains valid if $\Char(k) = 2$; see~Theorem~\ref{thm.main2}(a).
\end{remark}

\section{Explanation of the entries in Table~\ref{table1}}
\label{sect.table}
 
Throughout this section we will assume that the base field $k = \bbC$ is the field of complex numbers.
For the first row of the table, we used the following results:

\begin{itemize}
\item
$\ed(\Alt_4) = \ed(\Alt_5) = 2$, see~\cite[Theorem 6.7(b)]{buhler-reichstein},

\item
$\ed(\Alt_6) = 3$, see~\cite[Proposition 3.6]{serre-cremona},

\item
$\ed(\Alt_7) = 4$, see \cite[Theorem 1]{duncan},

\item
$\ed(\Alt_{n+4}) \geqslant \ed(\Alt_n) + 2$ for any $n \geqslant 4$, see \cite[Theorem 6.7(a)]{buhler-reichstein},

\item
$\ed(\Alt_n) \leqslant \ed(\Sym_n) \leqslant n - 3$ for any $n \geqslant 5$; see Lemma~\ref{lem.subgroup} 
and~\cite[Theorem 6.5(c)]{buhler-reichstein}.
\end{itemize}

The values of $\ed(\widetilde{\Alt}_n; 2)$ in the second row of Table~\ref{table1} are given by Theorem~\ref{thm.main1}(c).

In the third row, 

\begin{itemize}
\item
$\ed(\widetilde{\Alt}_4) = 2$ by Corollary~\ref{cor.main1}(a).

\item
To show that $\ed(\widetilde{\Alt}_5) = 2$, combine the inequalities 
$\ed(\widetilde{\Alt}_5) \geqslant \ed(\widetilde{\Alt}_5; 2) \geqslant 2$
of Theorem~\ref{thm.main1}(c) and 
$\ed(\widetilde{\Alt}_5) \leqslant 2$ of Theorem~\ref{thm.main1}(a). Alternatively, see~\cite[Lemma 2.5]{prokhorov}.

\item
$\ed(\widetilde{\Alt}_6) = 4$ by \cite[Proposition 2.7]{prokhorov}.

\item
To show that $\ed(\widetilde{\Alt}_7) = 4$, note that 
$\ed(\widetilde{\Alt}_7) \geqslant 4$ because $\widetilde{\Alt}_7$ contains $\widetilde{\Alt}_6$ and 
$\ed(\widetilde{\Alt}_7) \leqslant 4$ because $\widetilde{\Alt}_7$ has a faithful $4$-dimensional representation;
see~\cite[Corollary 2.1.4]{prokhorov}.

\item
The values of $\ed(\widetilde{\Alt}_8) = 8$ and $\ed(\widetilde{\Alt}_{16}) = 128$ are taken from
Corollary~\ref{cor.main1}(a). 

\item
When $9 \leqslant n \leqslant 15$ the range of values for $\ed(\widetilde{\Alt}_n)$ is given by
the inequality
\[ \ed(\widetilde{\Alt}_n) \leqslant \ed(\widetilde{\Alt}_n) + \ed(\Alt_n) \leqslant  \ed(\widetilde{\Alt}_n) + n - 3; \]
see Theorem~\ref{thm.main1}(e).
\end{itemize}

\section*{Acknowledgements} The authors are grateful to Eva Bayer-Fl\"uckiger, Vladimir Chernousov, Alexander Merkurjev, Jean-Pierre Serre, Burt Totaro, Alexander Vishik, and Angelo Vistoli for helpful comments.

\bibliographystyle{amsalpha}
\bibliography{spingroup}
\end{document}